\documentclass[12pt]{amsart}
\usepackage{amssymb}
\usepackage{geometry}
\usepackage{tikz-cd}
\usepackage{stmaryrd}
\usepackage{amsmath}
\usepackage{amsthm}
\usepackage{bbm}
\usepackage{enumitem}
\usepackage{lmodern}

\usepackage{hyperref}
\usepackage[capitalize,nameinlink]{cleveref}
\usepackage[backend=biber,style=alphabetic]{biblatex}
\DeclareMathOperator{\Hom}{Hom}

\DeclareMathOperator{\Ind}{Ind}

\DeclareMathOperator{\supp}{supp}

\newtheorem{theorem}{Theorem}[section]
\newtheorem{lemma}[theorem]{Lemma}
\newtheorem{proposition}[theorem]{Proposition}
\newtheorem{corollary}[theorem]{Corollary}

\theoremstyle{definition}
\newtheorem{definition}[theorem]{Definition}

\theoremstyle{remark}

\begin{document}
\title{On Multilinear Forms for Mod $p$ Representations of $\mathrm{GL}_2(\mathbb{Q}_p)$}
\author{Yikun Fan}
\maketitle
\begin{abstract}
    Motivated by Prasad's study of trilinear forms for complex representations
    \cite{trilinear}, we investigate the space of $G$-invariant linear forms on tensor products of irreducible admissible representations of $G = \mathrm{GL}_2(\mathbb{Q}_p)$ over $\overline{\mathbb{F}}_p$. 
    Our main result is a complete vanishing theorem: 
    for any $n \ge 1$ and 
    $n$ infinite-dimensional irreducible admissible representations $\pi_1,\dots,\pi_n$
     of $G$,
    \[
    \Hom_G(\pi_1 \otimes \cdots \otimes \pi_n, \mathbbm{1}) = 0.
    \]
    A refined version holds for 
    $B^+:=\begin{pmatrix}
          p^{\mathbb{Z}} & \mathbb{Q}_p \\
          0 & 1
          \end{pmatrix}$
          -invariant forms when at least one $\pi_i$ is supersingular. 
          The proof proceeds by a detailed analysis of certain subgroups, reducing the problem from $G$ to $B^+$ and ultimately to 
          the representation theory of $\mathbb{Z}_p$.
          We also deduce partial extensions of the result to $\mathrm{GL}_2(F)$ for finite extensions $F/\mathbb{Q}_p$.
    \end{abstract}
\begin{section}{Introduction}
\end{section}
Let \( p \) be a prime number, \( F \) be a finite extension of \( \mathbb{Q}_p \), \( G := \mathrm{GL}_2(F) \), 
It was proved in \cite[Theorem 1.1, Theorem 1.2]{trilinear} that
\begin{proposition}
Let $\pi_1,\pi_2,\pi_3$ 
be three admissible irreducible
representations of $G$ over $\mathbb{C}$, then
\[\dim_{\mathbb{C}}\Hom_{G}(\pi_1\otimes\pi_2\otimes \pi_3,\mathbbm{1})\leq 1\]
\end{proposition}
Inspired by this result, 
the main purpose of this note is to study a generalized version of the analogous question
 in
 $mod\ p$ setting. Because of the significant differences between the 
 $mod\ p$ and complex representation theories of $G$, 
 both the answer and the proof are completely different from \cite{trilinear}.
\par From now, all representations, 
without further explanation, are over the algebraically closed field \( k := \overline{\mathbb{F}}_p \).
 We first recall that, 
it was due to \cite{classification} that
the irreducible admissible $mod\ p$ representations of $G$ are classified as
    \begin{enumerate}
        \item principal series $\operatorname{Ind}_B^G\left(\theta_1 \otimes \theta_2\right)$ with $\theta_1 \neq \theta_2$
        \item $1$-dimensional representations $\theta \circ \operatorname{det}$
        \item special series $\mathrm{Sp} \otimes(\theta \circ \mathrm{det})$
        \item supersingular representations
    \end{enumerate}
    where $\mathrm{Sp}$ is defined by the exact sequence 
    \begin{equation*}
        0\to \mathbbm{1}\to \operatorname{Ind}_B^G\mathbbm{1}\to \mathrm{Sp}\to 0
    \end{equation*}
     and $\theta, \theta_1, \theta_2$ denote smooth characters $F^* \rightarrow k^*$.
    There are no non-trivial isomorphisms between representations of different types.
    \par Armed with this classification, we can now state our main results, which demonstrate a fundamental contrast with the complex case by establishing a universal vanishing phenomenon.
    \begin{theorem}\label{ssB}
        Suppose that $F = \mathbb{Q}_p$, and let $\pi_1$ be a supersingular representation, 
        $\pi_2$ be a smooth representation of 
        $G = \mathrm{GL}_2(\mathbb{Q}_p)$ over $\overline{\mathbb{F}}_p$. Define
     \[B^+ := \left\{ \begin{pmatrix}
         x & y \\
         0 & 1
         \end{pmatrix} \,\middle|\, x\in p^{\mathbb{Z}},\ y \in \mathbb{Q}_p \right\}.
         \] 
     Then
    \begin{displaymath}
        \Hom_{B^{+}}(\pi_1\otimes\pi_2,\mathbbm{1})=0
    \end{displaymath}
    \end{theorem}

\begin{theorem}\label{gens}
    Let $F$ be a finite extension of  
    $\mathbb{Q}_p$, $n\in\mathbb{N}_{>0}$ and $\pi_1,\cdots ,\pi_n$ be $n$ 
    irreducible admissible representations of 
    $G=GL_2(F)$ over $\overline{\mathbb{F}}_p$, 
     each one is either a principal series or a 
special series. Then
\begin{displaymath}
    \Hom_{G}(\bigotimes_{i=1}^n\pi_i,\mathbbm{1})=0
\end{displaymath}
\end{theorem}

As an immediate consequence of the preceding theorems and the classification of irreducible admissible representations, we obtain:

\begin{corollary}\label{pmain}
    Suppose that $F=\mathbb{Q}_p$, $n\in\mathbb{N}_{>0}$ and $\pi_1,\pi_2,\pi_3,\cdots ,\pi_n$ are $n$ 
    infinite-dimensional admissible irreducible representations of 
    $G=\mathrm{GL}_2(\mathbb{Q}_p)$ over $\overline{\mathbb{F}}_p$. Then
\begin{displaymath}
\Hom_G(\bigotimes_{i=1}^n\pi_i,\mathbbm{1})=0
\end{displaymath}
\end{corollary}

 Here we summarize the content of each sections.
In \cref{not}, we fix the notations and some other 
preliminaries used in this note.
\par In \cref{vansec}, 
we prove several key lemmas about the $mod\ p$ representation theory of $\mathbb{Z}_p$ 
via purely linear algebraical discussions, and
gives the explicit definition (\cref{vandef}) of 
\emph{vanishing representation} of $B^+$. 
By reducing to $\mathbb{Z}_p$ representation theory, we establish

\begin{lemma}[Vanishing Lemma]\label{van}
    Let $\pi_1$ be a vanishing representation of 
    $B^+$, $\pi_2$ be a smooth representation of $B^+$ 
    and $\chi:T\to \overline{\mathbb{F}}_p^*$ be a character, 
     then 
    \begin{align*}\Hom_{B^+}(\pi_1\otimes \pi_2,\chi)=0
    \end{align*}
\end{lemma}

 which helps us
reduce the proof of \cref{ssB} to demonstrating that 
supersingular representations of $\mathrm{GL}_2(\mathbb{Q}_p)$ are vanishing as $B^+$-representation.
\par In \cref{cons}, we find an explicit vanishing 
$B^+$-subrepresentation in
supersingular representations, principal series, and lift of special series 
respectively. 
\par For $\pi$ a principal series or special series, 
$\pi$ is a quotient of $\Ind_B^G(\theta_1\otimes\theta_2)$ for some characters 
$\theta_1,\theta_2:F^*\to k^*$, 
we prove that the $B^+$-subrepresentation
\[
\bigl\{\,f\in \Ind_B^G(\theta_1\otimes\theta_2) \mid \supp(f)\subseteq B
\begin{pmatrix}
    0 & 1
    \\ 1 & 0
\end{pmatrix}
B \bigr\}
\]
of $\Ind_B^G(\theta_1\otimes\theta_2)$ is vanishing.
\par For a supersingular representation $\pi$,
 we 
 find a subspace $M_{\pi}$ of $\pi$ such that 
 $(\pi,M_{\pi})$ is a vanishing pair under the assumption $F=\mathbb{Q}_p$, based on the results in \cite{Breuil_2003} and \cite{paskunas1}. 
 This construction, together with \cref{van}, directly imply \cref{ssB}.
\par In \cref{modpmain}, we conclude the proof of \cref{gens}
via the vanishing lemma and the construction of vanishing pair in \cref{cons}.
\section*{Acknowledgements}

This work was written as part of my ALGANT 
master's thesis at the University of Duisburg-Essen, 
under the supervision of Prof.\ Vytautas Paškūnas. 
I am deeply grateful to him for his guidance 
throughout the project. 
I also thank Prof.\ Jan Kohlhaase very much 
for his valuable suggestions, 
in particular for his help in shaping the 
final form of \cref{zprep} and its proof.
\section{Notations}\label{not}
Let \( F \) be a finite extension of \( \mathbb{Q}_p \). Denote by \( \mathcal{O} \) the ring of integers of \( F \), let \( \varpi \) be a fixed uniformizer of \( \mathcal{O} \), and let \( \mathfrak{m} \) be the maximal ideal of \( \mathcal{O} \).

Define the matrices:
\[
s := \begin{pmatrix} 0 & 1 \\ 1 & 0 \end{pmatrix}, \quad
t := \begin{pmatrix} \varpi & 0 \\ 0 & 1 \end{pmatrix}.
\]
We introduce the following subgroups of \( G := \operatorname{GL}_2(F) \):
\begin{itemize}
    \item The \textbf{center}:
        \[
        Z := \left\{ \begin{pmatrix} x & 0 \\ 0 & x \end{pmatrix} \,\middle|\, x \in F^* \right\}.
        \]
    \item The \textbf{diagonal torus}:
        \[
        T := \left\{ \begin{pmatrix} x & 0 \\ 0 & y \end{pmatrix} \,\middle|\, x, y \in F^* \right\}.
        \]
    \item The \textbf{Borel subgroup} (upper triangular matrices):
        \[
        B := \left\{ \begin{pmatrix} x & z \\ 0 & y \end{pmatrix} \,\middle|\, x, y \in F^*,\, z \in F \right\}.
        \]
    \item The \textbf{unipotent radical} of \( B \):
        \[
        U := \left\{ \begin{pmatrix} 1 & z \\ 0 & 1 \end{pmatrix} \,\middle|\, z \in F \right\}.
        \]
    \item A subgroup \( B^+ \subset B \):
        \[
        B^+ := \left\{ \begin{pmatrix} x & y \\ 0 & 1 \end{pmatrix} \,\middle|\, x \in \varpi^{\mathbb{Z}},\, y \in F \right\}.
        \]
    \item The \textbf{unipotent part of the Iwahori subgroup}:
        \[
        U_0 := \operatorname{GL}_2(\mathcal{O}) \cap U = \left\{ \begin{pmatrix} 1 & b \\ 0 & 1 \end{pmatrix} \,\middle|\, b \in \mathcal{O} \right\}.
        \]
    \item The \textbf{\( \mathbb{Z}_p \)-parts of \( U_0 \)}:
        \[
        U_1 := \left\{ \begin{pmatrix} 1 & b \\ 0 & 1 \end{pmatrix} \,\middle|\, b \in \mathbb{Z}_p \right\}.
        \]
\end{itemize}

For each integer \( n \geq 1 \), the \textbf{principal congruence subgroup} of level \( n \) is:
\[
K_n := \left\{ g \in \operatorname{GL}_2(\mathcal{O}) \,\middle|\, g \equiv I_2 \pmod{\varpi^n} \right\},
\]
where \( I_2 \) is the \( 2 \times 2 \) identity matrix.

We recall the following fundamental decomposition:\par

\begin{proposition}\cite[(5.2) Bruhat decomposition]{llcgl2}: $G=BsB\coprod B=BsU\coprod B$
\end{proposition}

Let $\Gamma$ be a locally profinite group with a closed subgroup $H$, $M$ be a bi-$H$-invariant open subset of $\Gamma$ 
(i.e., $HMH = M$), and $(\pi,V)$ be a smooth representation of $H$. 
We define 
\[
    \Ind_H^\Gamma\pi
    :=\left\{\,f:\Gamma\to V\;\middle|\;
    \begin{array}{l}
    f(hg)=\pi(h)f(g)\quad\text{for all }h\in H,\ g\in \Gamma,\\[4pt]
    \text{there exists an open subgroup }U_f\subseteq \Gamma\ \text{such that}\\[2pt]
    f(gu)=f(g)\ \text{for all }g\in \Gamma,\ u\in U_f
    \end{array}
    \right\},
\]
endowed with the $\Gamma$-action
\[
    (g\cdot f)(g') := f(g'g), \quad \text{for all }g,g'\in\Gamma.
\]

and define 
\[
\Ind_H^M\pi := \{f \in \Ind_H^\Gamma\pi \mid \supp(f) \subseteq M\},
\]
which is naturally a smooth $H$-subrepresentation of $\Ind_H^\Gamma\pi$ under the restricted $H$-action
\[
    (h\cdot f)(g) := f(gh), \quad \text{for all }h\in H, g\in\Gamma.
\]

Finally, given a smooth representation $(\pi, V)$ of $\mathbb{Z}_p$, 
we write $\pi(a)$ for the action of $a \in \mathbb{Z}_p$ on $V$.  
When no confusion is likely, we suppress $\pi$ from the notation and 
denote this action simply by
\[
[a] : V \to V, \qquad v \mapsto \pi(a)v,
\]
its action on $V$. 
In particular, we define the difference operator
\[
\Delta : V \to V, \qquad v \mapsto [1]v - [0]v.
\]
\begin{section}{The Vanishing Pair, Vanishing Representation and Vanishing Lemma}\label{vansec}
  The primary goal for this section is a 
  \emph{Vanishing Lemma} (\cref{van}),
   which reduces the non-existence of $B^+$-invariant multilinear forms to the construction of certain subspaces in each representation. Its proof is founded on a study of smooth $\mathbb{Z}_p$-representations
   (\cref{zprep}, \cref{trad}, \cref{sequ}), 
   which establishes the condition $\Delta V = V$ as a sufficient criterion for the vanishing of $\mathbb{Z}_p$-invariant forms, as well as technical criteria for its verification.

  This principle is abstracted into the notion of a \emph{vanishing pair} (\cref{vandef}), forming the framework for its application to representations of $B^+$.
  \begin{lemma}\label{zprep}
    Let $V_1$ and $V_2$ be two smooth representations of \(\ \mathbb{Z}_p \) 
    such that \(\Delta V_1=V_1\), then 
    \(\Delta(V_1\otimes V_2)=(V_1\otimes V_2)\). In particular, 
    \[\Hom_{\mathbb{Z}_p}(V_1\otimes V_2,\mathbbm{1})=0\]
  \end{lemma}
  \begin{proof}
    We prove the first assertion, as the second follows immediately from it.
    
    For each $n\in\mathbb{N}$, denote 
    \[\ker\left(\Delta^n\right):=\{v\in V_2\mid\Delta^nv=0\}.\]
    Since \(V_2\) is smooth, we have the filtration
    \[
    V_2 = \bigcup_{n=1}^{\infty} \ker\left(\Delta^n\right).
    \]
    It therefore suffices to show that for every \(n \in \mathbb{Z}_{>0}\),
    \[
    \Delta\left(V_1\otimes \ker(\Delta^n)\right) = V_1\otimes \ker(\Delta^n).
    \]
    We proceed by induction on \(n\).
  
    \textbf{Base Case (\(n = 1\)):}
    Let \(v \in V_1\) and \(u \in \ker(\Delta)\). By hypothesis, there exists \(v' \in V_1\) such that \(\Delta v' = v\). Then
    \[
    \begin{aligned}
    \Delta(v'\otimes u) &= [1]v'\otimes [1]u - [0]v'\otimes [0]u \\
    &= [1]v'\otimes u - v'\otimes u \\
    &= (\Delta v')\otimes u = v\otimes u.
    \end{aligned}
    \]
    Hence, \(v\otimes u \in \Delta(V_1\otimes \ker(\Delta))\), and the base case holds.

    \textbf{Inductive Step:}
    Assume the statement holds for some \(n \geq 1\). Consider the exact sequence of \(\mathbb{Z}_p\)-representations:
    \[
    0 \to \ker(\Delta^n) \to \ker(\Delta^{n+1}) \to Q \to 0,
    \]
    where \(Q := \ker(\Delta^{n+1})/\ker(\Delta^n)\). Tensoring with \(V_1\) yields the exact sequence:
    \[
    V_1\otimes \ker(\Delta^n) \to V_1\otimes \ker(\Delta^{n+1}) \to V_1\otimes Q \to 0.
    \]
    
    Note that \(\Delta Q = 0\) by construction. By the same argument as in the base case, we have
    \[
    \Delta(V_1\otimes Q) = V_1\otimes Q.
    \]
    
    Now, given any element \(x \in V_1\otimes
     \ker(\Delta^{n+1})\), its image \(\bar{x}\) in \(V_1\otimes Q\) satisfies \(\bar{x} = \Delta y\) for some \(y \in V_1\otimes Q\). Let \(y' \in V_1\otimes \ker(\Delta^{n+1})\) be a lift of \(y\). Then \(x - \Delta y'\) lies in \(V_1\otimes \ker(\Delta^n)\). By the induction hypothesis,
    \[
    x - \Delta y' = \Delta z \quad \text{for some }
     z \in V_1\otimes \ker(\Delta^n).
    \]
    Therefore, \(x = \Delta(y' + z)\in \Delta(V_1\otimes
    \ker(\Delta^{n+1}))\)
    , completing the induction.
\end{proof}

    \begin{proposition}\label{trad}
            Let \(V_1, \ldots, V_n\) be smooth infinite-dimensional representations of \(\mathbb{Z}_p\). 
            If \(\dim_k(V_i^{\mathbb{Z}_p}) = 1\) for all \(1 \leq i \leq n\), then
            \[
            \bigoplus_{i=1}^n V_i \subseteq \Delta\left(\bigoplus_{i=1}^n V_i\right).
            \]
    \end{proposition}
    \begin{proof}
        It suffices to show that \(\Delta V_i = V_i\) for each \(1 \leq i \leq n\). 
        We fix an index \(i\) and prove this equality.
       \par For each $j\in\mathbb{Z}_{>0}$,
$\Delta$ restricts to an endomorphism of 
$V_i^{p^j\mathbb{Z}_p}$ satisfying that
\begin{displaymath}
(\Delta|_{V_i^{p^j\mathbb{Z}_p}})^{p^j}=
([p^j]-[0])|_{V_i^{p^j\mathbb{Z}_p}}=0
\end{displaymath}
and
\begin{displaymath}
\dim_k(\ker(\Delta|_{V_i^{p^j\mathbb{Z}_p}}))=
\dim_k(V_i^{\mathbb{Z}_p})=1
\end{displaymath}
while $V_i^{p^{j-1}\mathbb{Z}_p}\subseteq V_i^{p^{j}\mathbb{Z}_p}$
equals to the subspace 
\[\ker([p^{j-1}]-[0])=\ker((\Delta|_{V_i^{p^j\mathbb{Z}_p}})^{p^{j-1}})\]

By standard linear algebra (considering the Jordan form), the dimension \(d_j := \dim_k(V_i^{p^j\mathbb{Z}_p})\) satisfies \(d_j \leq p^j\), 
and there exists a basis \(\{v_1, \ldots, v_{d_j}\}\) of \(V_i^{p^j\mathbb{Z}_p}\) such that:
\begin{itemize}
\item \(V_i^{p^{j-1}\mathbb{Z}_p} = \operatorname{span}\langle v_1, \ldots, v_{d_{j-1}} \rangle\),
\item \(\Delta|_{V_i^{p^j\mathbb{Z}_p}}\) is represented by a nilpotent Jordan block matrix, i.e. of the form
\begin{displaymath}
    \begin{pmatrix}
        0 & 1 & 0 & \cdots & 0 & 0 \\
        0 & 0 & 1 & \cdots & 0 & 0 \\
        0 & 0 & 0 & \cdots & 0 & 0 \\
        \vdots & \vdots & \vdots & \ddots & \vdots & \vdots \\
        0 & 0 & 0 & \cdots & 0 & 1 \\
        0 & 0 & 0 & \cdots & 0 & 0 \\
        \end{pmatrix}
    \end{displaymath}
\end{itemize}
In particular, for \(1 \leq l < d_j\), we have \(v_l = 
\Delta v_{l+1}\). 
If \(d_{j-1} < d_j\), then:
\begin{equation}\label{buzhidao}
\Delta V_i^{p^j\mathbb{Z}_p} = \operatorname{span}\langle v_1, \ldots, v_{d_j-1} \rangle \supseteq \operatorname{span}\langle v_1, \ldots, v_{d_{j-1}} \rangle = V_i^{p^{j-1}\mathbb{Z}_p}.
\end{equation}
It follows from the smoothness of \(V_i\), the condition \(\dim_k(V_i) = \infty\), and the bound \(d_j \leq p^j\) for all \(j \in \mathbb{N}\) that
\[
V_i = \bigcup_{\substack{j \in \mathbb{N} \\ V_i^{p^{j-1}\mathbb{Z}_p} \subsetneq V_i^{p^j\mathbb{Z}_p}}} V_i^{p^{j-1}\mathbb{Z}_p}.
\]
This, together with \eqref{buzhidao}, yields
\[
V_i \subseteq \Delta V_i.
\]
This completes the proof for each \(i\), and the proposition follows.

\end{proof}
\begin{lemma}\label{sequ}
    Let \(V\) be a smooth representation of \(\mathbb{Z}_p\). Suppose that for every \(v \in V^{\mathbb{Z}_p}\), there exists a sequence of vectors \(\{v_n \in V \mid n \in \mathbb{N}\}\) such that
    \[
    v_0 = 0, \quad v_1 = v, \quad \text{and} \quad \Delta v_{n+1} = v_n \quad \text{for all } n \in \mathbb{N}.
    \]
    Then \(\Delta V = V\).
    \end{lemma}
    
    \begin{proof}
    Since \(V\) is smooth, we have \(V = \bigcup_{i=1}^{\infty} \ker\left(\Delta^i\right)\). It therefore suffices to show that
    \[
    \ker\left(\Delta^i\right) \subseteq \Delta V
    \]
    for every \(i \in \mathbb{Z}_{>0}\). We will prove this by induction on \(i\).
    
    \noindent\textbf{Base Case (\(i = 1\)):}
    Let \(u \in \ker\Delta\). By hypothesis, there exists a sequence \(\{v_n\}\) with \(v_1 = u\) and \(\Delta v_{n+1} = v_n\) for all \(n \in \mathbb{N}\). In particular, \(\Delta v_2 = v_1 = u\). Hence, \(u \in \Delta V\).
    
    \noindent\textbf{Inductive Step:}
    Assume that \(\ker\left(\Delta^k\right) \subseteq 
    \Delta V\) for some \(k \in \mathbb{Z}_{>0}\).
     Let \(u \in \ker\left(\Delta^{k+1}\right)\). If \(u \in \ker\left(\Delta^k\right)\), then \(u \in \Delta V\) by the induction hypothesis.
    
    Now, suppose \(u \notin \ker\left(\Delta^k\right)\). Then \(w := \Delta^k u\) is a nonzero vector in \(\ker\Delta\). By the hypothesis on \(V\), there exists a sequence \(\{v_n \in V \mid n \in \mathbb{N}\}\) such that
    \[
    v_0 = 0, \quad v_1 = w, \quad \text{and} \quad \Delta v_{n+1} = v_n \quad \text{for all } n \in \mathbb{N}.
    \]
    
    Now compute that the vector \(u - v_{k+1}\) satisfies
    \[
    \Delta^k (u - v_{k+1}) = \Delta^k u - \Delta^k v_{k+1} = w - v_1 = w - w = 0.
    \]
    Thus, \(u - v_{k+1} \in \ker\left(\Delta^k\right)\). By the induction hypothesis,
    \[
    u - v_{k+1} \in \Delta V.
    \]
    Furthermore, from the sequence property, we have \(v_{k+1} = \Delta v_{k+2} \in \Delta V\). Therefore, we conclude that
    \[
    u = (u - v_{k+1}) + v_{k+1} \in \Delta V.
    \]
    This completes the induction step, and hence the proof.
    \end{proof}

\begin{definition}\label{vandef}
    Let $\pi$ be an infinite-dimensional smooth representation of 
    $B^+$ over $k$, if there exists a  
    $U_1$-subrepresentation $W$ of $\pi$ such that
    \begin{enumerate}
      \item $\pi = \sum_{n \in \mathbb{N}} t^{n} W$;
      \item $\Delta W=W$ as $\mathbb{Z}_p\cong U_1$-representation.
    \end{enumerate}
    Then we call the pair $(\pi, W)$ a \emph{vanishing pair} and $\pi$ a \emph{vanishing representation}.
  \end{definition}

  \begin{lemma}[Vanishing Lemma]
    Let $\pi_1$ be a vanishing representation of 
    $B^+$, $\pi_2$ be a smooth representation of $B^+$ 
    and $\chi:T\to k^*$ be a character, 
     then 
    \begin{align*}\Hom_{B^+}(\pi_1\otimes \pi_2,\chi)=0
    \end{align*}
\end{lemma}
\begin{proof} Let $W_1$ be a linear subspace of $\pi_1$ 
    such that $(\pi_1,W_1)$ is a vanishing pair.
    \par It follows from \cref{zprep} that 
    \begin{align}\label{res}
        \Hom_{U_1}\left( W_1\otimes\pi_2, \mathbbm{1} \right)
        &= 0.
    \end{align}
     \textbf{Claim}: The restriction map 
            \begin{align*}
                \Hom_{B^+}\left(\pi_1\otimes\pi_2, \chi\right) 
                &\longrightarrow \Hom_{U_1}\left(W_1\otimes \pi_2, \mathbbm{1} \right) \\
                \phi &\longmapsto \phi|_{W_1\otimes\pi_2} \notag
            \end{align*}
        is injective.
    \par    The assertion follows directly from the claim and \cref{res}.
\\ \textbf{Proof of Claim:}   Since 
 \( \pi_1 = \sum_{j \in\mathbb{N}} t^{j}W_1 \), we get
            
        \[
        \pi_1\otimes \pi_2 = \sum_{j \in\mathbb{N}} t^{j} \left( W_1\otimes\pi_2 \right)
        \]
         Let \( v \in \pi_1\otimes\pi_2\) and 
         \( \phi \in \Hom_{B^+}\left(\pi_1\otimes\pi_2, \chi \right) \).
        Then $v =\sum_{i=1}^mt^{n_i}w_i$ for some 
        \(m\in\mathbb{N}\) and \(n_i\in\mathbb{N} ,w_i
         \in W_1\otimes\pi_2 \) for each $1\leq i\leq m$, and
        \[
        \phi(v) = \phi(\sum_{i=1}^mt^{n_i}w_i) = 
        \sum_{i=1}^m\chi(t^{n_i})\phi(w_i)
        \]
        Hence \( \phi \) is determined by its restriction to \( W_1\otimes\pi_2 \).
\end{proof}
\end{section}
\begin{section}{Construction of Vanishing Pair}\label{cons}
    With the discussions in the last section, particularly the vanishing lemma, 
    we reduce the proof of \cref{ssB}, \cref{gens} and 
    \cref{pmain} into the construction of a vanishing pair 
    in every infinite-dimensional admissible irreducible representations of $G$.
In this section, we construct the vanishing pair in principal series, 
special series 
and supersingular representations respectively.
\subsection{Principal Series and Special Series Case}
Since both principal series and special series are quotients of 
\(\Ind_B^G\chi\) for some character $\chi:T\to k^*$, 
it suffices to perform the construction for arbitrary $\Ind_B^G\chi$.
\begin{proposition}\label{prin}
        Let $\chi$ be an arbitrary character of $T$, then 
      \begin{displaymath}
      (\Ind_{B}^{BsB}\chi,W_\chi:=\Ind_{B}^{BsU_0}\chi)
      \end{displaymath}
        is a vanishing pair and \(\Ind_{B}^{BsB}\chi\)
     is vanishing. \end{proposition}
      \begin{proof}
          \begin{enumerate}
              \item Observe that
              \[
                s \begin{pmatrix} 1 & x \\ 0 & 1 \end{pmatrix} t^{-1}
                = \begin{pmatrix} 0 & 1 \\ \frac{1}{\varpi} & x \end{pmatrix}
                = \begin{pmatrix} 1 & 0 \\ 0 & \frac{1}{\varpi} \end{pmatrix}
                  s \begin{pmatrix} 1 & \varpi x \\ 0 & 1 \end{pmatrix}
                \in BsU_0,
              \]
              for any $x\in\mathcal{O}$,
              hence
              \begin{equation}\label{t-1-inv}
                BsU_0 t^{-1} \subseteq BsU_0.
              \end{equation}
               Let $f \in \Ind_B^{BsB} \chi$.
             Since $f$ is locally constant, compactly supported 
               modulo $B$ and right-invariant under some 
               compact open subgroup $K_n$ of $G$, it is supported on finitely many double cosets $Bg_i K_n \subseteq BsB$. 
          Thence there exists an $l\geq 1$ such that
      $\cup_{i=1}^l Bg_iK_n\supseteq \supp(f)$.
          
              Since the image of $\cup_{i=1}^l Bg_iK_n$ in 
              $B \backslash G$ is compact and
              \[
                \{ BsU_0 t^k\mid k\geq 0\}
              \]
              forms an open cover of $BsB$, 
              \begin{displaymath}
                  \supp(f)\subseteq \cup_{i=1}^l Bg_iK_n\subseteq \bigcup_{i=0}^m 
                  BsU_0t^i
              \end{displaymath}
      for some $m\geq 0$.

      \par Furthermore, it follows from \cref{t-1-inv} that 
      \[\bigcup_{i=0}^m 
      BsU_0t^i=BsU_0t^m\]
      hence $\supp(f)\subseteq BsU_0t^m$ and 
      $\supp(t^{m}f)=\supp(f)t^{-m}$ is contained in $BsU_0$.
      Therefore, $f \in t^{-m} W_\chi$, completing the proof of (1).
          
              \item It is easy to see that 
              $W_{\chi}$ is infinite-dimensional, stable under the action of $U_1\cong \mathbb{Z}_p$
             , we now show that $\Delta W_\chi=W_\chi$. 
\par Via the natural identification \(B\backslash G\cong \mathbb{P}^1\), 
we may identify $W_\chi=\Ind_B^{BsU_0}\chi$ with the space of locally constant functions on 
$\mathcal{O}$.
              
\par It follows from the elementary algebra that
 there exists a sequence of functions 
 \(\{\phi_n(t):\mathcal{O}/\mathfrak{m}\to\overline{\mathbb{F}}_p|n\in\mathbb{N}\}\) such that 
 $\phi_0(x)=1$
and 
 $\phi_{n+1}(x+1)-\phi_{n+1}(x)=\phi_{n}(x)$ for any $n\in\mathbb{N},x\in \overline{\mathbb{F}}_p$. 
\par For any $f \in W_\chi^{U_1}$,
 define 
\begin{align*}
    f_n: \mathcal{O}\to \overline{\mathbb{F}}_p
    \\ x\mapsto \phi_n(\overline{x})f(x)
\end{align*}
where $\overline{x}$ denotes the image of $x$ under the natural quotient 
\[\mathcal{O}\to \mathcal{O}/\mathfrak{m}\]
By the $\mathbb{Z}_p$-invariance of $f$, the sequence \(\{f_n|n\in\mathbb{N}\}\) satisfies that 
\begin{align*}
    (\Delta f_{n+1})(x)&=f_{n+1}(x+1)-f_{n+1}(x)
   =\phi_{n+1}(\overline{x+1})f(x+1)-\phi_{n+1}(\overline{x})f(x)
   \\ &=(\phi_{n+1}(\overline{x+1})-\phi_{n+1}(\overline{x}))f(x)
    =\phi_n(\overline{x})f(x)=f_n(x)
\end{align*}
this, together with \cref{sequ}, 
completes the proof of (2).\qedhere
            \end{enumerate}
      \end{proof}
      \subsection{On Supersingular Case For $\mathrm{GL}_2(\mathbb{Q}_p)$}
    In this subsection, we fix $F=\mathbb{Q}_p$
and complete the construction of vanishing pair for supersingular representations in $F=\mathbb{Q}_p$ case.  
This construction cannot be carried over directly to general case due to the lack of 
a classification of supersingular representations. 
Particularly the \cref{pro} was built only for $F=\mathbb{Q}_p$ case.

   \par Let 
   \[I_1:=\left(\begin{array}{cc}
    1+\mathfrak{m} & \mathbb{Z}_p \\
    \mathfrak{m} & 1+\mathfrak{m}
    \end{array}\right)\] 
and $(\pi,V)$ be a supersingular representation of $G$. We quote two results about the supersingular representation
\begin{proposition}[{\cite[Theorem 3.2.4, Corollary 4.1.4]{Breuil_2003}}]
    The space $\pi^{I_1}$ of $I_1$-invariant vectors in $V$ is $2$-dimensional over 
$k$.
\end{proposition}
\begin{theorem}[{\cite[Theorem 1.1]{paskunas2}}]\label{Birre}
The restriction $\pi|_B$ is an irreducible representation of $B$.
\end{theorem}

Denote
\[
    M_{\pi}:=\langle\begin{pmatrix} p^{\mathbb{N}} & \mathbb{Z}_p
        \\ 0 & 1
    \end{pmatrix}
        \pi^{I_1}\rangle
\]

\begin{proposition}[{\cite[Lemma 4.6, Proposition 4.7, Lemma 4.8, Lemma 4.10]{paskunas1}}]\label{pro}
There exists a basis $v_{\sigma},v_{\tilde\sigma}$ of $\pi^{I_1}$ such that
\begin{displaymath}
    M_\sigma:=\left\langle\left(\begin{array}{cc}
        p^{2 \mathbb{N}} & \mathbb{Z}_p \\
        0 & 1
        \end{array}\right) v_\sigma\right\rangle,
        M_{\tilde{\sigma}}:=\left\langle\left(\begin{array}{cc}
            p^{2 \mathbb{N}} & \mathbb{Z}_p \\
            0 & 1
            \end{array}\right) v_{\tilde{\sigma}}\right\rangle 
\end{displaymath}
satisfies that
\begin{enumerate}
        \item $M,M_{\sigma},M_{\tilde{\sigma}}$ are all stable under $U_0$.
        \item $M_\sigma^{U_0}$ and $M_{\tilde{\sigma}}^{U_0}$ are both of dimension $1$ over $k$.
        \item $tM_{\sigma}\subseteq M_{\tilde{\sigma}}$, $tM_{\tilde{\sigma}}\subseteq M_{\sigma}$
        \item $M_{\pi}=M_\sigma\oplus M_{\tilde{\sigma}}$
    \end{enumerate}
\end{proposition}

\begin{proposition}\label{ss}
    Every supersingular representation of $G$ is a vanishing representation of 
    $B^+$.
\end{proposition}
\begin{proof}
    Let $\pi$ be an arbitrary  supersingular 
    representation of $G$, it suffices to show that $(\pi,M_{\pi})$ is a vanishing pair.
    (2) follows directly from \cref{pro} and \cref{trad}
    , thus suffices to prove (1) here.
\par It follows from \cref{Birre} that $\pi=\langle BM_\pi\rangle$.
Since $B=\{t^{n}|n\in\mathbb{N}\}Z\begin{pmatrix} p^{\mathbb{N}} & \mathbb{Z}_p
    \\ 0 & 1
\end{pmatrix}
$, $M_\pi$ is by definition stable under the action of $\begin{pmatrix} p^{\mathbb{N}} & \mathbb{Z}_p
    \\ 0 & 1
\end{pmatrix}$
and all elements of $Z$ act as scalar multiple on $\pi$, 
we conclude that
\begin{displaymath}
\pi=\langle\{t^{n}|n\in\mathbb{N}\}M_\pi\rangle
=\langle t^{n}M_\pi|n\in\mathbb{N}\rangle=
\sum_{n\in\mathbb{N}}t^n M_\pi.
\end{displaymath}
\end{proof}
The \cref{ssB} follows immediately from \cref{ss} and the vanishing lemma.
\end{section}

\begin{section}{The Proof of \cref{gens}}\label{modpmain} 

    Since both principal series and 
    special series are quotients of $\Ind_B^G\chi$ for some character
     $\chi: T \to k^*$, \cref{gens} follows
     directly
    from the following proposition.

\begin{proposition}\label{strong}
    Let $n \in \mathbb{Z}_{>0}$, 
    $\chi_i: T \to k^*$ be a character 
    and $\pi_i:=\Ind_B^G\chi_i$ for each $1\leq i\leq n$.
     Then
    \[
    \Hom_G\left( \bigotimes_{i=1}^n \pi_i, \chi \right) = 0
    \]
    holds for any character $\chi : T \to k^*$.
\end{proposition}
    \begin{proof}
        We proceed by induction on $n$. 
      \par  \textbf{Base case ($n=1$):} 
      If $\pi_1$ is irreducible, the assertion 
      $\Hom_G(\pi_1,\chi)=0$ is straightforward.
       \par If $\pi_1$ is reducible,
       by assumption and the classification of irreducible admissible representations of $G$, 
       this happens precisely when $\pi_1 = \Ind_B^G\chi_1$ for some character $\chi_1: T \to k^*$ that factors through the determinant map, i.e., $\chi_1 = \theta \circ \det$ for some character $\theta: F^* \to k^*$. 

       In this case, $\Ind_B^G\chi_1$ admits a non-split short exact sequence:
       \[
       0 \to \chi_1 \to \Ind_B^G\chi_1 \to \mathrm{Sp} \otimes \chi_1 \to 0,
       \]
    since $\mathrm{Sp}\otimes \chi_1$ is infinite-dimensional, 
    irreducible and the exact sequence is non-split, the assertion 
    \(\Hom_G(\pi_1,\chi)=\Hom_G(\Ind_B^G\chi_1,\chi)=0\)
    thence follows.
      \par \textbf{Inductive step:} Assume the proposition holds for some $m \geq 1$. Let $\pi_1, \dots, \pi_{m+1}$ be $m+1$ representations satisfying the hypothesis.  
       \par Let $\chi$ be an arbitrary character of $T$. It follows from \cref{prin} and \cref{ssB} that
        \begin{align}\label{sub}
            \Hom_{B^+}\left(  \Ind_B^{BsB}\chi_{m+1}\otimes (\bigotimes_{i=1}^m\pi_i) , \chi \right) = 0.
        \end{align}
        Furthermore, as a $B^+$-representation, there is an isomorphism:
        \begin{align}\label{dec}
            \left( \bigotimes_{i=1}^{m+1} \pi_i \right) 
            \Bigg/ \left( \Ind_B^{BsB}\chi_{m+1}\otimes 
            (\bigotimes_{i=1}^m\pi_i) \right)
             \cong  \chi_{m+1} \otimes \left( \bigotimes_{\substack{1 \leq i \leq m}} \pi_i \right)  .
        \end{align}
        By the induction hypothesis,
       \begin{equation}\label{quo}
        \begin{aligned}
           & \Hom_G\left( \left( \chi_{m+1} \otimes \left( \bigotimes_{\substack{1 \leq i \leq m}} \pi_i \right) \right) , \chi \right) 
           \\ &= \Hom_G\left( \left( \bigotimes_{\substack{1 \leq i \leq m}} \pi_i \right) , \chi_{m+1}^{-1}\chi \right) \\
            &= 0.
        \end{aligned}
    \end{equation}
        The assertion therefore follows directly from the quotient 
        \cref{dec}, together with the vanishing results in \cref{sub} and \cref{quo}.
    \end{proof}
    \end{section}
\printbibliography
\end{document}